\documentclass{amsart}
\usepackage{breqn}
\begin{document}
\newtheorem{thm}{Theorem}[section]
\newtheorem{prop}[thm]{Proposition}
\newtheorem{claim}[thm]{Claim}
\newtheorem{corollary}[thm]{Corollary}
\newtheorem{lemma}[thm]{Lemma}
\newtheorem{porism}[thm]{Porism}

\theoremstyle{definition}  
\newtheorem{definition}[thm]{Definition}
\newtheorem{example}[thm]{Example}
\newtheorem{examples}[thm]{Examples}
\newtheorem{exercise}[thm]{Exercise}
\newtheorem{note}[thm]{Note}
\newtheorem{remark}[thm]{Remark}
\newtheorem{notation}[thm]{Notation}
\newtheorem{question}[thm]{Question}
\newtheorem{observation}[thm]{Observation}
\newtheorem{fact}[thm]{Fact}
\newtheorem*{vlth}{van Lambalgen's Theorem}

\newtheorem*{construct}{Construction}
\newtheorem*{verify}{Verification}

\newcommand{\sgsp}{\renewcommand{\baselinestretch}{1}\tiny\normalsize}
\newcommand{\sqsp}{\renewcommand{\baselinestretch}{1.5}\tiny\normalsize}
\newcommand{\dbsp}{\renewcommand{\baselinestretch}{2}\tiny\normalsize}

\raggedbottom
\tolerance=3000
\hbadness=10000
\hfuzz=1.5pt
\setcounter{tocdepth}{1}
\setcounter{secnumdepth}{2}

\newcommand{\margnote}[1]{\mbox{}\marginpar{\hspace{0pt}#1}}

\newcommand{\e}{ \{e\} }
\newcommand{\reals}{2^\N}

\newcommand{\strings}{2^{<\N}}
\newcommand{\nat}{\in\N}
\newcommand{\power}{\mathcal{P}(\N)}
\newcommand{\restr}{\!\!\restriction\!\!}
\newcommand{\st}{\;|\;}
\newcommand{\conv}{\!\!\downarrow}
\newcommand{\dvrg}{\!\!\uparrow}
\newcommand{\diff}{\!\setminus\!}
\newcommand{\inv}{^{-1}}
\newcommand{\str}[1]{\langle #1 \rangle}

\newcommand{\M}{\mathcal{M}}
\newcommand{\cR}{\mathcal{R}}
\newcommand{\B}{\mathcal{B}}
\newcommand{\F}{\mathbb{F}}
\newcommand{\Op}{\mathcal{O}}
\newcommand{\Fnc}{\mathcal{F}}
\newcommand{\Consts}{\mathcal{C}}
\newcommand{\LB}{\mathcal{S}_{\operatorname{Banach}}}

\newcommand{\concat}{^\smallfrown}

\newcommand{\eqT}{\equiv_T}
\newcommand{\eqwtt}{\equiv_{wtt}}
\newcommand{\eqtt}{\equiv_{tt}}

\newcommand{\leqT}{\leq_T}
\newcommand{\leqwtt}{\leq_{wtt}}
\newcommand{\leqtt}{\leq_{tt}}

\newcommand{\classR}{{\sf R}}
\newcommand{\classS}{{\sf S}}
\newcommand{\classC}{\mathcal{C}}
\newcommand{\classD}{\mathcal{D}}

\newcommand{\K}{$\mathcal{K}$}

\newcommand{\re}{c.e.\ }
\newcommand{\nre}{$n$-change}
\newcommand{\dre}{d.r.e.\ }
\newcommand{\Nre}{$\N$-change}
\newcommand{\fre}{$f$-change}
\newcommand{\gre}{$g$-change}

\newcommand{\field}[1]  {\mathbb{#1}} 
\newcommand{\R}         {\field{R}}
\newcommand{\N}         {\field{N}}
\newcommand{\Z}         {\field{Z}}
\newcommand{\C}         {\field{C}}
\newcommand{\Q}         {\field{Q}}
\newcommand{\LL}        {\field{L}}
\newcommand{\D}{\mathbb{D}}
\newcommand{\XTC}{\hat{\mathbb{C}}}
\newcommand{\cat}{^\frown}

\newcommand{\dom}{\operatorname{dom}}
\newcommand{\ran}{\operatorname{ran}}
\newcommand{\diam}{\operatorname{diam}}
\newcommand{\normal}{n}
\newcommand{\Log}{\operatorname{Log}}
\newcommand{\Arg}{\operatorname{Arg}}
\renewcommand{\Re}{\operatorname{Re}}
\renewcommand{\Im}{\operatorname{Im}}
\newcommand{\Int}{\operatorname{Int}}
\newcommand{\Ext}{\operatorname{Ext}}
\newcommand{\ray}[1]{\overrightarrow{#1}}
\newcommand{\QED}{\box}
\newcommand{\const}{\mbox{\emph{Const.}} }
\newcommand{\norm}[1]{\| #1 \|}
\newcommand{\supp}{\operatorname{supp}}
\newcommand{\Ind}{\mathbf{1}}
\newcommand{\mathquote}[1]{\mbox{`}#1\mbox{'}}

\newcommand{\stra}{\mathcal{A}}
\newcommand{\strb}{\mathcal{B}}
\newcommand{\strc}{\mathcal{C}}
\newcommand{\classk}{\mathcal{K}}

\newcommand{\zerovec}{\mathbf{0}}

\renewcommand{\d}{\mathbf{d}}

\numberwithin{equation}{section}

\title{Degrees of and lowness for isometric isomorphism}

\author[Franklin]{Johanna N.Y.\ Franklin}
\address[Franklin]{Department of Mathematics \\ Room 306, Roosevelt Hall \\ Hofstra University \\ Hempstead, NY 11549-0114 \\ USA}
\email{johanna.n.franklin@hofstra.edu}
\urladdr{http://people.hofstra.edu/Johanna\_N\_Franklin/}
\thanks{The first author was supported in part by Simons Foundation Collaboration Grant \#420806.}

\author[McNicholl]{Timothy H.\ McNicholl}
\address[McNicholl]{Department of Mathematics\\ Iowa State University\\ Ames, Iowa 50011}
\email{mcnichol@iastate.edu}
\thanks{The second author was supported in part by Simons Foundation Collaboration Grant \#317870.}

\begin{abstract}
We contribute to the program of extending computable structure theory to the realm of metric structures by investigating lowness for isometric isomorphism of metric structures.  We show that lowness for isomorphism coincides with lowness for isometric isomorphism and with lowness for isometry of metric spaces.  We also examine certain restricted notions of lowness for isometric isomorphism with respect to fixed computable presentations, and, in this vein, we obtain classifications of the degrees that are low for isometric isomorphism with respect to the standard copies of certain Lebesgue spaces.
\end{abstract}
\maketitle

\section{Introduction}\label{sec:intro}

While lowness---the idea that an oracle is useless in a particular context---has appeared in several contexts in computability theory over the past 50 years, it only made its way into computable structure theory in the past few years with Franklin and Solomon's results on lowness for isomorphism \cite{fs-lowim}. Franklin and Solomon defined a degree $\mathbf{d}$ to be low for isomorphism if, whenever there is a $\mathbf{d}$-computable isomorphism between two computably presented structures $\stra$ and $\strb$, there is already a computable isomorphism between $\stra$ and $\strb$ and thus the information contained in $\mathbf{d}$ is unnecessary in this context. This is clearly a degree-theoretic property, and the class of Turing degrees with this property has proven difficult to characterize. 

However, one may also define lowness for isomorphism for a class of structures $\classk$: whenever there is a $\mathbf{d}$-computable isomorphism between two computably presented structures $\stra$ and $\strb$ in a given class $\classk$, there is a computable isomorphism between $\stra$ and $\strb$. In Suggs's thesis, he considered classes of various types of equivalence structures, linear orders, and shuffle sums and was frequently able to achieve a full characterization of lowness for isomorphism for these particular classes \cite{suggs}.

These results are all formulated for classes of countable algebraic structures.  
Here, we turn our attention towards analysis and focus on metric structures.  Roughly speaking, these structures consist of a complete metric space together with collections of operations, functionals, and constants.  Examples are Banach spaces, Hilbert spaces, probability spaces, and $C^*$ algebras.  
The model theory of these structures has been investigated extensively via continuous logic (see \cite{Ben-Yaacov.Berenstein.Henson.Usvyatsov.2008}).   

Recently,  a program to adapt the framework of computable structure theory to the continuous setting, that is, to metric structures, has emerged (see \cite{Melnikov.2013}, \cite{Melnikov.Ng.2014}, \cite{Melnikov.Nies.2013}, \cite{Nies.Solecki.2015}, \cite{McNicholl.2017}, \cite{Clanin.McNicholl.Stull.2019}, \cite{Brown.McNicholl.2019}).   We contribute to this direction by introducing the study of lowness for isometric isomorphism of metric structures.  
We begin by considering metric structures in general in Section \ref{sec:metric.gen}.  We find, perhaps not surprisingly, that lowness for isomorphism (of countable algebraic structures) and lowness for isometric isomorphism coincide in Section \ref{sec:metric}.  

We then follow a direction parallel to that pursued by Suggs and 
consider some specific classes of metric structures.  We first find that lowness for isomorphism and lowness for isometry of metric spaces coincide.  We then proceed to examine Banach spaces, a class of metric structures that has enjoyed a long history of investigation in analysis as well as many interactions with mathematical logic.  We find that every degree of isomorphism is a degree of isometric isomorphism for Banach spaces.  We find that proving the converse of this statement, or more generally obtaining a classification of these degrees, appears to be a difficult task.  We discuss some of these difficulties in Section \ref{sec:banach}.  

As a possible first step towards obtaining a classification of the degrees of isometric isomorphism of Banach spaces, 
we then proceed to narrow our focus even further and consider the Lebesgue spaces.  Our motivation for doing so is that 
these spaces, in particular the sequence spaces $\ell^p$, are often used in the constructions of examples of Banach spaces.  
In addition, $\ell^1$ is universal among separable spaces in that every separable Banach space is a quotient of $\ell^1$.  
We succeed in classifying the degrees 
that are low for isometric isomorphism for the standard presentations of these spaces, that is, the degrees that are useless for computing an isometric isomorphism of the standard presentation onto some other presentation.

Finally, in Section \ref{sec:conclusion}, we state several questions and conjectures that naturally arise from these results. 

We begin with some preliminaries regarding metric structures and their presentations and, then, the fundamental definitions of lowness for (and degrees of) isometric isomorphism.

\section{Background and preliminaries}\label{sec:back.prelim}

\subsection{Metric structures and their presentations}\label{sec:back.prelim::subsec:metric.struct}

We begin by more formally defining the concept of a metric structure and the associated concepts of metric signature and interpretation.  Our definitions are essentially the same as that found in standard sources such as \cite{Ben-Yaacov.Berenstein.Henson.Usvyatsov.2008}.  
The main difference is that we do not require our metric structures to be bounded.  
A more minor difference is that we replace predicates with the somewhat broader class of functionals. 

Let $\F$ denote the field of scalars.  This can be either $\R$ or $\C$.

\begin{definition}\label{def:metric.str}
A \emph{metric structure} is a quintuple $\M = (U,d, \Op, \Fnc, \mathcal{C})$ with the following properties.
\begin{enumerate}
	\item $(U,d)$ is a complete metric space.

	\item For each $T \in \Op$, there is a positive integer $n$ so that $T$ is a uniformly continuous $n$-ary operation on $U$.
	
	\item For each $f \in \Fnc$, there is a positive integer $n$ so that $f$ is a uniformly continuous $n$-ary functional on $U$; i.e., $f : U^n \rightarrow \F$ and is uniformly continuous. 
	
	\item $\mathcal{C} \subseteq U$.
\end{enumerate}
\end{definition}

We remark that every countable algebraic structure can be represented by a metric structure by employing the discrete metric and regarding the characteristic functions of the relations as functionals.

If $\M$ is a metric structure, let $|\M|$ denote the set of all points of $\M$ (i.e., the \emph{universe} of $\M$).

\begin{definition}\label{def:metric.sig}
We define a \emph{metric signature} to be a quintuple 
$(\overline{\Op}, \overline{\Fnc}, \overline{\Consts}, \eta, \Delta)$ where
\begin{enumerate}
	\item $\overline{\Op}$, $\overline{\Fnc}$, $\overline{\Consts}$ are pairwise disjoint sets of symbols,
	
	\item $\eta : \overline{\Op} \cup \overline{\Fnc} \cup \overline{\Consts} \rightarrow \N$, $\eta$ is positive on $\overline{\Op} \cup \overline{\Fnc}$, and $\eta(\overline{c}) = 0$ for each $\overline{c} \in \overline{\Consts}$, and 
	
	\item $\Delta : (\overline{\Op} \cup \overline{\Fnc}) \times \N \rightarrow \N$.
\end{enumerate}
\end{definition}

Suppose $\mathcal{S} = (\overline{\Op}, \overline{\Fnc}, \overline{\Consts}, \eta, \Delta)$ is a metric signature.  
We refer to the symbols in $\overline{\Op}$, $\overline{\Fnc}$, $\overline{\Consts}$ as the 
\emph{operation symbols}, \emph{functional symbols}, and \emph{constant symbols} of $\mathcal{S}$ respectively.  
We call $\eta(\gamma)$ the \emph{arity} of $\gamma$.  
The function $\Delta_\phi(n) = \Delta(\phi,n)$ is the \emph{modulus} of $\phi$.

Suppose $(X_1, d_1)$ and $(X_2, d_2)$ are metric spaces and $f : X_1^n \rightarrow X_2$.  
Recall that a function $g : \N^n \rightarrow \N$ is a \emph{modulus of continuity} for $f$ if 
$d(f(p_1, \ldots, p_n), f(q_1, \ldots, q_n)) < 2^{-k}$ whenever 
$\max_j d(p_j, q_j) \leq 2^{-g(k)}$.

We can say now what it means for a metric structure to interpret a signature $\mathcal{S}$.

\begin{definition}
A metric structure $\M$ is an \emph{interpretation} of $\mathcal{S}$ if there is a map $\mathcal{I}$ that satisfies the following conditions. 
\begin{enumerate}
	\item For each $n$-ary operation symbol $\phi$ of $\mathcal{S}$, $\mathcal{I}(\phi)$ is an $n$-ary 
	operation of $\M$.  
	
	\item For each $n$-ary functional symbol $\phi$ of $\mathcal{S}$, $\mathcal{I}(\phi)$ is an $n$-ary functional of $\M$.  
	
	\item For each constant symbol $\kappa$ of $\mathcal{S}$, $\mathcal{I}(\kappa)$ is a point of $\M$.
	
	\item $\Delta_\phi$ is a modulus of continuity for $\phi$.  
\end{enumerate}
We denote $\mathcal{I}(\phi)$ by $\phi^\M$.
\end{definition}

Conversely, if $\M$ is an interpretation of $\mathcal{S}$, then we say $\mathcal{S}$ is a signature of $\M$. Let $\mathcal{K}_{\mathcal{S}}$ denote the class of all interpretations of $\mathcal{S}$ (the class of \emph{$\mathcal{S}$-structures}).  

\begin{definition}
Suppose $\M_0$ and $\M_1$ are interpretations of a metric signature $\mathcal{S}$, and let $F : |\M_0| \rightarrow |\M_1|$. 
We say $F$ is an \emph{isomorphism} if it is homeomorphic and satisfies the following.  
\begin{enumerate}
	\item For each $n$-ary operation symbol $\overline{T}$ of $\mathcal{S}$ and all $p_1, \ldots, p_n \in |\M_0|$,  
	$F(\overline{T}^{\M_0}(p_1, \ldots, p_n)) = \overline{T}^{\M_1}(F(p_1), \ldots, F(p_n))$.
	
	\item For each $n$-ary functional symbol $\overline{\phi}$ of $\mathcal{S}$ and all $p_1, \ldots, p_n \in |\M_0|$, 
	$F(\overline{\phi}^{\M_0}(p_1, \ldots, p_n)) = \overline{\phi}^{\M_1}(F(p_1), \ldots, F(p_n))$.  
	
	\item For each constant symbol $\overline{c}$ of $\mathcal{S}$, 
	$F(\overline{c}^{\M_0}) = \overline{c}^{\M_1}$.
\end{enumerate}
\end{definition}

A map $\Phi : |\M_0| \rightarrow |\M_1|$ is \emph{isometric} (or an \emph{isometry}) if 
it preserves distances.  We are primarily interested in isometric isomorphisms as they preserve both the metric and the algebraic structure.\footnote{In the literature on the model theory of metric structures, the term ``isomorphism" is used for maps that preserve the metric and algebraic structures.  In keeping with the terminology of functional analysis, we prefer to use the term ``isomorphism" for maps that preserve the topological and algebraic structures.}

We turn to presentations of metric structures which we will use to define computability on these structures.  Our approach is an adaptation of an idea that goes back to Pour-El and Richards \cite{Pour-El.Richards.1989}.  We will need the following.  

\begin{definition}\label{def:generated}
Suppose $\M$ is a metric structure.
\begin{enumerate}
	\item If $S \subseteq |\M|$, then the \emph{subspace generated by $S$} is the smallest closed subset of $|\M|$
	that is closed under every operation of $\M$.
	
	\item A sequence $(p_n)_{n \in \N}$ \emph{generates} $\M$ if 
	$\{p_n\ :\ n \in \N\}$ generates $|\M|$.
\end{enumerate} 
\end{definition}

\begin{definition}\label{def:presentation}
Suppose $\M$ is a metric structure.  A \emph{presentation} of $\M$ is a pair $(\M, (p_n)_{n \in \N})$ such that 
$(p_n)_{n \in \N}$ generates $\M$.  If $\M^\# = (\M, (p_n)_{n \in \N})$ is a presentation of $\M$, we call $p_n$ the $n$-th \emph{distinguished point} of $\M$.
\end{definition}

Thus, a presentation of a metric structure is entirely defined by specifying the sequence of its distinguished points.  
Note that we do not require the distinguished points to be dense.  

\begin{definition}\label{def:rational.pts}
Suppose $\M^\#$ is a presentation of a metric structure $\M$.  
\begin{enumerate}
	\item The \emph{rational points} of $\M^\#$ are the points in the subspace generated by the distinguished points of $\M^\#$.
	
	\item A \emph{rational open ball} of $\M^\#$ is an open ball of $\M$ whose center is a rational point of $\M^\#$ and whose radius is a positive rational number.
\end{enumerate}
\end{definition}

We now turn to the computability of presentations.  This necessitates a brief discussion of the computability of metric signatures. 
Fix a metric signature $\mathcal{S} = (\overline{\mathcal{O}}, \overline{\mathcal{F}}, \overline{\mathcal{C}}, \eta,\Delta)$.  A \emph{presentation} of $\mathcal{S}$ is a pair $(\mathcal{S}, \nu)$ where $\nu$ 
maps $\N$ onto the symbols of $\mathcal{S}$.  A presentation $(\mathcal{S}, \nu)$ is \emph{computable} if it meets the following criteria.
\begin{enumerate}
	\item $\nu^{-1}[X]$ is computable for each $X \in \{\overline{\mathcal{O}}, \overline{\mathcal{F}}, \overline{\mathcal{C}}\}$.  
	
	\item $\eta \circ \nu$ is computable.
	
	\item $\Delta_{\nu(n)}$ is computable uniformly in $n \in \nu^{-1}[\overline{\mathcal{O}} \cup \overline{\mathcal{F}}]$.
\end{enumerate}
We observe that any two computable presentations of a metric signature are computably isomorphic.  That is, 
if $(\mathcal{S}, \nu)$ and $(\mathcal{S}, \nu')$ are computable presentations, then there is a computable permutation $\pi$ of $\N$ so that $\nu \circ \pi = \nu'$.  Thus, if a metric signature has a computable presentation, we identify that signature with any one of its computable presentations, and we simply call the signature \emph{computable}.  The key feature of such signatures is that if $\M^\#$ is a presentation of a metric structure that has a computable signature, then it is possible to effectively number its rational points and balls.  This numbering allows us to define 
computable points and maps as follows.

\begin{definition}\label{def:comp.pt}
Suppose $\M^\#$ is a presentation of a metric structure that has a computable signature, and let 
$d$ denote the metric of $\M$.  
A point $p$ of $\M^\#$ is a \emph{computable point of $\M^\#$} if there is an algorithm 
that, given any $k \in \N$, produces a rational point $p'$ of $\M^\#$ so that $d(p,p') < 2^{-k}$.
\end{definition}

\begin{definition}\label{def:comp.map}
Suppose $\M_0^\#$ and $\M_1^\#$ are presentations of metric structures with computable signatures, and let
$\Phi : |\M_0| \rightarrow |\M_1|$.  We say $\Phi$ is a \emph{computable map of $\M_0^\#$ into $\M_1^\#$} if there is an algorithm $P$ that satisfies the following two criteria.
\begin{itemize}
	\item Given a (code of a) rational ball $B_0$ of $\M_0^\#$, $P$ either does not 
	halt or produces a rational ball $B_1$ of $\M_1^\#$ so that $\Phi[B_0] \subseteq B_1$.  
	
	\item If $U$ is a neighborhood of $\Phi(p)$, then there is a rational ball $B_0$ of 
	$\M_0^\#$ so that $p \in B_0$ and $P(B_0) \subseteq U$.
\end{itemize}
\end{definition}

\begin{definition}\label{def:comp.pres}
Suppose $\M^\#$ is a presentation of a metric structure $\M$ that has a computable signature.  
We say $\M^\#$ is \emph{computable} if it satisfies the following conditions.  
\begin{enumerate}
	\item The metric of $\M$ is computable on the rational points of $\M^\#$.  
That is, if $d$ denotes the metric of $\M$, then there is an algorithm that, given any two rational points $p_1$, $p_2$ of 
$\M^\#$ and a $k \in \N$, computes a rational number $q$ so that $|q - d(p_1, p_2)| < 2^{-k}$. 

	\item For every $n$-ary functional $F$ of $\M^\#$ and all rational points $p_1, \ldots, p_n$ of $\M^\#$, 
	$F(p_1, \ldots, p_n)$ is computable uniformly in $F, p_1, \ldots, p_n$.  That is, there is an algorithm that, given 
	$F$,  $p_1, \ldots, p_n$ and $k \in \N$ as input, produces a rational number $q$ so that 
	$|F(p_1, \ldots, p_n) - q| < 2^{-k}$.
\end{enumerate}
\end{definition}

If $e,e'$ are indices of the algorithms referenced in Definition \ref{def:comp.pres}, then we refer to 
$\langle e,e' \rangle$ as an \emph{index} of $\M^\#$.

The criteria in the following theorem can often be used to reduce the computability of a function between presentations of metric spaces to its computability on the rational points.  This can be useful in demonstrating the computability of such a function in that the mystery of producing an algorithm that operates on neighborhoods can be sidestepped for the more familiar setting of computing on individual points. 

\begin{thm}\label{thm:comp.map.equiv}
Suppose $\M_0^\#$ and $\M_1^\#$ are presentations of metric structures with computable signatures, and let
$\Phi : |\M_0| \rightarrow |\M_1|$.  Then $\Phi$ is a computable map of $\M_0^\#$ into $\M_1^\#$ if both of the following hold. 
\begin{enumerate}
	\item $\Phi$ is computable on the rational points of $\M_0^\#$.  That is, for every rational point $p$ of $\M_0^\#$, 
	$\Phi(p)$ is a computable point of $\M_1^\#$ uniformly in $p$.
	
	\item There is a computable modulus of continuity for $\Phi$.
\end{enumerate}
\end{thm}

Both Definition \ref{def:comp.pres} and Theorem \ref{thm:comp.map.equiv} relativize.  

A metric structure may have a presentation that is designated as standard.  In such a case, the structure and its standard presentation are identified.  Standard presentations are always computable.

\subsection{Lowness for and degrees of isomorphism and isometry}\label{sec:back.prelim::subsec:lowness}

Throughout this section, we assume $\mathcal{S}$ is a computable metric signature and that 
all metric structures considered are interpretations of $\mathcal{S}$.  If $\phi$ is either an operation or functional symbol of $\mathcal{S}$, then, as before, we let $\Delta_\phi$ denote the modulus of continuity assigned to $\phi$ by $\mathcal{S}$.

We now formally define lowness for isometric isomorphism.  Since we are considering this concept both as it pertains to a particular structure as well as in general, we will need to break the usual definition down into three sublevels: lowness for isometric isomorphism for a given computable presentation of a structure, lowness for isometric isomorphism for a given structure, and, finally, lowness for isometric isomorphism for a class of structures. 

\begin{definition}\label{def:low.ii}
Let $\mathbf{d}$ be a nonzero Turing degree.
\begin{enumerate}
	\item Suppose $\mathcal{M}^\#$ is a computable presentation of a metric structure $\M$.  We say that $\mathbf{d}$ is \emph{low for $\mathcal{M}^\#$ isometric isomorphism}
if every computable presentation of $\mathcal{M}$ that is $\mathbf{d}$-isometrically isomorphic to $\mathcal{M}^\#$ is also computably isometrically isomorphic to $\mathcal{M}^\#$. 

	\item Suppose $\mathcal{M}$ is a computably presentable metric structure.  We say $\mathbf{d}$ is \emph{low for $\mathcal{M}$ isometric isomorphism} if it is low for $\mathcal{M}^\#$ isometric isomorphism whenever 
	$\mathcal{M}^\#$ is a computable presentation of $\mathcal{M}$.
	
	\item Suppose $\mathcal{K}$ is a class of computably presentable metric structures.  We say $\mathbf{d}$ is \emph{low for isometric isomorphism of $\mathcal{K}$-structures} if $\mathbf{d}$ is low for $\mathcal{M}$ isometric isomorphism for every $\mathcal{M} \in \mathcal{K}$.  
	
	\item We say $\mathbf{d}$ is \emph{low for isometric isomorphism} if it is low for $\mathcal{M}$ isometric isomorphism for every computably presentable structure $\mathcal{M}$.  
\end{enumerate}
\end{definition}

We will also discuss degrees of isometric isomorphism, introduced by McNicholl and Stull in \cite{McNicholl.Stull.2019}. We present this concept here at two levels: first, at the level of a degree of isomorphism for a pair of computable presentations of a structure, and then at the level of the degree of isomorphism of a single computable presentation of a structure.

\begin{definition}\label{def:deg.ii}
Let $\mathcal{M}$ be a metric structure.
\begin{enumerate}
	\item Suppose $\mathcal{M}^\#$ and $\mathcal{M}^+$ are computable presentations of $\mathcal{M}$.  The \emph{degree of isometric isomorphism} of $(\mathcal{M}^\#, \mathcal{M}^+)$ is the least powerful Turing degree that computes an isometric isomorphism of $\mathcal{M}^\#$ onto $\mathcal{M}^+$. 
	
	\item If among all computable presentations of $\mathcal{M}$ one, say $\mathcal{M}^\#$, is designated as standard, and if $\mathcal{M}^+$ is any computable presentation of $\mathcal{M}$, then the \emph{degree of isometric isomorphism of $\mathcal{M}^+$} is the degree of isometric isomorphism of $(\mathcal{M}^\#, \mathcal{M}^+)$.  
\end{enumerate}
\end{definition}

Our interest in degrees of isometric isomorphism stems from the following observation.  Suppose $\mathcal{M}^\#$ has the property that the degree of isometric isomorphism for $(\mathcal{M}^\#, \mathcal{M}^+)$ is defined for every computable $\mathcal{M}^+$.  Then the degrees that are low for $\mathcal{M}^\#$ isometric isomorphism are precisely those that do not bound these degrees of isometric isomorphism.

As noted above, every countable algebraic structure can be represented as a metric structure.  Therefore, when applying the terminology  defined in this section to such structures, we omit ``isometry."   

\section{Lowness for isometric isomorphism}\label{sec:metric.gen}

Throughout this section, we assume $\mathcal{S}$ is a computable metric signature and that 
all structures considered are interpretations of $\mathcal{S}$.  The main result of this section is the following.

\begin{thm}\label{thm:low.ii}
A Turing degree is low for isomorphism if and only if it is low for isometric isomorphism.
\end{thm}

We base the proof of Theorem \ref{thm:low.ii} on the following lemma which will be useful later as well.  
The lemma and its proof are adaptations of ideas from \cite{ft-lowpaths}.

\begin{lemma}\label{lm:pi.01}
Let $\M^\#$ and $\M^+$ be computable presentations of a metric structure with signature $\mathcal{S}$.  
Then there is a $\Pi_1^0$ class $\cR \subseteq \N^\N$ so that
for every Turing degree $\d$, $\d$ computes a point in $\cR$ if and only if 
$\d$ computes an isometric isomorphism of $\M^\#$ onto $\M^+$.
\end{lemma}

\begin{proof}
 Let $x_j$ denote the $j$-th rational point of $\M^\#$, and let $y_j$ denote the $j$-th rational point of $\M^+$.  
 For every $n$-ary operation $T$ of $\M$, fix computable maps $\zeta_T$ and $\zeta_T'$ from $\N^n$ into $\N$ so that for all $j_1, \ldots, j_n \in \N$, $x_{\zeta_T(j_1, \ldots, j_n)} = T(x_{j_1}, \ldots, x_{j_n})$ and $T(y_{j_1}, \ldots, y_{j_n}) = y_{\zeta'_T(j_1, \ldots, j_n)}$.  Furthermore, we can choose these maps so that they are computable uniformly in $T$.
For each constant $c$ of $\M$, fix computable maps $\zeta_c$ and $\zeta_c'$ so that for all 
$c,j$,
\[
\max\{d(x_{\zeta_c(j)}, c), d(y_{\zeta'_c(j)}, c)\} < 2^{-j}.
\]
Again, we can choose these maps so that they are computable uniformly in $c$.

We define $\cR$ to be the set of all $(f,g) \in (\N^{\N \times \N})^2$, recoded as elements of $\N^\N$, that satisfy the following conditions. 
\begin{enumerate}
	\item $\max\{d(y_{f(m,n)}, y_{f(m,n+1)}), d(x_{g(m,n)}, x_{g(m,n+1)})\} \leq 2^{-(n+1)}$.  \label{itm:def.R.Cauchy.cms}
	
	\item $|d(x_m, x_{m'}) - d(y_{f(m,n)}, y_{f(m',n')})| \leq 2^{-n} + 2^{-n'}$.
\label{itm:def.R.isom.cms}
	
	\item For all $m,n,n' \in \N$, 
	\[
	\max\{d(x_m, x_{g(f(m,n), n')}), d(y_m, y_{f(g(m,n), n')})\} \leq 2^{-n} + 2^{-n'}.
	\]
	\label{itm:def.R.inv.cms}
	
	\item For every $n$-ary operation $T$ of $\M$ and all $m,k,j_1, \ldots, j_n, k_1, \ldots, k_n \in \N$, 
	\[
	d(y_{f(\zeta_T(j_1, \ldots, j_n), k)}, y_{\zeta'_T(f(j_1,k_1), \ldots, f(j_n, k_n))}) \leq 2^{-(k+1)} + 2^{-m}
	\]
provided $\Delta_T(m) \leq \min_s k_s + 1$ for $s$ between $1$ and $n$.	
		\label{itm:def.R.OpsPres}
	
	\item For every $n$-ary functional $F$ of $\M$ and all $m, j_1, \ldots, j_n, k_1, \ldots, k_n \in \N$, 
	\[
	|F(x_{j_1}, \ldots, x_{j_n}) - F(y_{f(j_1, k_1)}, \ldots, y_{f(j_n, k_n)})| \leq 2^{-m}
	\]
	provided $\Delta_F(m) \leq \min_s k_s + 1$  for $s$ between $1$ and $n$.
	\label{itm:def.R.Func.Pres}
	
	\item For each constant $c$ of $\M$ and all $j, k \in \N$, 
	\[
	d(y_{\zeta'_c(j)}, y_{f(\zeta_c(j), k)}) \leq 2^{-j+1} + 2^{-(k+1)}.
	\]
	\label{itm:def.R.Const.Pres}
\end{enumerate}
Thus, $\cR$ is $\Pi_1^0$.  \\

We first show that if $\d$ computes an isometric isomorphism $\Phi$ of $\M^\#$ onto $\M^+$, then 
$\d$ computes an $(f,g) \in \cR$.   Let 
$\Psi = \Phi^{-1}$; it follows that $\d$ computes $\Psi$ as well.
It also follows that $\d$ computes $f,g \in \N^{\N \times \N}$ so that for each 
$m \in \N$, $(x_{f(m,n)})_{n \in \N}$ and $(y_{g(m,n)})_{n \in \N}$ are strongly Cauchy sequences that converge to $\Phi(x_m)$ and $\Psi(y_m)$ respectively.  Thus, $f$ and $g$ satisfy (\ref{itm:def.R.Cauchy.cms}) and (\ref{itm:def.R.isom.cms}).  

Now we suppose for a contradiction that $(f,g)$ does not satisfy condition (\ref{itm:def.R.inv.cms}).  Without loss of generality, suppose 
$d(x_m, x_{g(f(m,n), n')}) > 2^{-n} + 2^{-n'}$.  
Since $(y_{f(m,n)})_{n \in \N}$ and $(x_{g(f(m,n),k )})_{k \in \N}$ are strongly Cauchy sequences that converge to $\Phi(x_m)$ and $\Psi(y_{f(m,n)})$ respectively, 
$d(\Phi(x_m), y_{f(m,n)}) \leq 2^{-n}$ and 
$d(x_{g(f(m,n), n')}, \Psi(y_{f(m,n)})) \leq 2^{-n'}$.  Since
$\Psi$ is an isometry, $d(\Psi\Phi(x_m), \Psi(y_{f(m,n)})) \leq 2^{-n}$.   
Thus, $d(\Psi\Phi(x_m), x_{g(f(m,n), n')}) \leq 2^{-n} + 2^{-n'}$.  
But $\Psi\Phi(x_m) = x_m$, so we have our contradiction.  

Now we consider condition (\ref{itm:def.R.OpsPres}).  
Let $T$ be an $n$-ary operation of $\M$, and let $m,k,j_1, \ldots, j_n,k_1, \ldots, k_n \in \N$ so that $m$ satisfies the condition.  
Since $\Delta_T$ is a modulus of continuity for $T$, 
\begin{dmath*}
d(y_{\zeta'_T(f(j_1, k_1), \ldots, f(j_n,k_n))}, T(\Phi(x_{j_1}), \ldots, \Phi(x_{j_n}))) \\
 =  d(T(y_{f(j_1, k_1)}, \ldots, y_{f(j_n, k_n)}), T(\Phi(x_{j_1}), \ldots, \Phi(x_{j_n}))) \\
  <   2^{-m}.
\end{dmath*}
Since $\Phi$ is an isomorphism, 
\begin{dmath*}
d(T(\Phi(x_{j_1}), \ldots, \Phi(x_{j_n})), y_{f(\zeta_T(j_1, \ldots, j_n), k)})  =  
d(T(\Phi(x_{j_1}), \ldots, \Phi(x_{j_n})), \Phi(T(x_{j_1}, \ldots, x_{j_n}))) + 
d(\Phi(T(x_{j_1}, \ldots, x_{j_n})), y_{f(\zeta_T(j_1, \ldots, j_n), k)}) 
 =  d(\Phi(T(x_{j_1}, \ldots, x_{j_n})), y_{f(\zeta_T(j_1, \ldots, j_n), k)}) \leq 2^{-(k+1)}.
\end{dmath*}

Next, we demonstrate that $(f,g)$ satisfies condition (\ref{itm:def.R.Func.Pres}). Suppose $F$ is an $n$-ary functional of $\M$, and let 
$m, j_1, \ldots, j_n, k_1, \ldots, k_n \in \N$ so that $\Delta_F(m) \leq \min_s k_s + 1$.  Then 
\begin{multline*}
|F(x_{j_1}, \ldots, x_{j_n}) - F(y_{f(j_1, k_1)}, \ldots, y_{f(j_n, k_n)})| \\
\leq |F(x_{j_1}, \ldots, x_{j_n}) - F(\Phi(x_{j_1}), \ldots, \Phi(x_{j_n}))| \\
+ |F(\Phi(x_{j_1}), \ldots, \Phi(x_{j_n})) - F(y_{f(j_1, k_1)}, \ldots, y_{f(j_n, k_n)})|.
\end{multline*}
Since $\Phi$ is an isomorphism, $F(x_{j_1}, \ldots, x_{j_n}) = F(\Phi(x_{j_1}), \ldots, \Phi(x_{j_n}))$.
Since $\Delta_F$ is a modulus of continuity for $F$, 
$|F(\Phi(x_{j_1}), \ldots, \Phi(x_{j_n})) - F(y_{f(j_1, k_1)}, \ldots, y_{f(j_n, k_n)})| \leq 2^{-m}$.\\

Finally, we show $(f,g)$ satisfies condition (\ref{itm:def.R.Const.Pres}). Let $c$ be a constant of $\M$, and let $j,k \in \N$.  
Then 
\begin{eqnarray*}
d(y_{\zeta'_c(j)}, y_{f(\zeta_c(j), k)}) & \leq & d(y_{\zeta'_c(j)}, c) + d(y_{f(\zeta_c(j), k)}, \Phi(x_{\zeta_c(j)})) + 
d(\Phi(x_{\zeta_c(j)}), c)\\
\ & \leq & 2^{-j} + 2^{-(k+1)} + d(\Phi(x_{\zeta_c(j)}), c).
\end{eqnarray*}
Since $\Phi$ is an isometric isomorphism, 
\[
d(\Phi(x_{\zeta_c(j)}), c) = d(\Phi(x_{\zeta_c(j)}), \Phi(c)) = d(x_{\zeta_c(j)}, c) \leq 2^{-j}.
\]
Thus, (\ref{itm:def.R.Const.Pres}) is satisfied, and so $(f,g)$ is a $\d$-computable pair in $\cR$.

Conversely, suppose $\d$ computes a pair $(f,g) \in \cR$.  Let $\Phi(x_m) = \lim_n x_{f(m,n)}$, and let 
$\Psi(y_m) = \lim_n y_{g(m,n)}$.  Note that by (\ref{itm:def.R.isom.cms}), $\Phi$ and $\Psi$ are well defined and have isometric extensions to $\M$.  We denote these extensions by $\Phi$ and $\Psi$ as well.  

We claim that $\Psi\Phi(x_m) = x_m$.  We suppose otherwise for a contradiction.  Then there exists $n_1 \in \N$ so that 
$d(x_m, \Psi(y_{f(m,n_1)})) > 2^{-n_1}$.  Thus, there also exists $n_2 \in \N$ so that 
$d(x_m, x_{g(f(m,n_1), n_2)}) > 2^{-n_1} + 2^{-n_2}$.  This contradicts 
condition (\ref{itm:def.R.inv.cms}).  

We similarly show $\Psi\Phi(y_m) = y_m$.  Thus, by continuity, $\Phi = \Psi^{-1}$.  

Since $\Phi$ and $\Psi$ are $\d$-computable on the rational points of $\M^\#$ and $\M^+$ respectively, and since they have a computable modulus of computability, they are $\d$-computable.

We now show $\Phi$ preserves the operations and functionals of $\M$.   To begin, suppose $T$ is an $n$-ary operation of $\M$.  
We must show that 
$\Phi(T(p_1, \ldots, p_n)) = T(\Phi(p_1), \ldots, \Phi(p_n))$ for all points $p_1, \ldots, p_n$ of $\M$.  
By continuity, it suffices to consider the case where each $p_j$ is a rational point of 
$\M^\#$.  So, for each $s \in \{1, \ldots, n\}$, let $j_s$ be an index so that $p_s = x_{j_s}$.  It then suffices to show that for each $\epsilon > 0$, 
\[
d(\Phi(T(x_{j_1}, \ldots, x_{j_n})), T(\Phi(x_{j_1}), \ldots, \Phi(x_{j_n}))) < \epsilon.
\]
We thus let $\epsilon > 0$ and choose $m \in \N$ so that $2^{-m}  < \epsilon$.  Choose $k \in \N$ so that $2^{-m} + 2^{-(k+1)} < \epsilon$ and $m \leq k+1$.  By (\ref{itm:def.R.OpsPres}), 
\[
d(y_{f(\zeta_T(j_1, \ldots, j_n), k)}, y_{\zeta'_T(f(j_1, k_1), \ldots, f(j_n, k_n))}) \leq 2^{-(k+1)} + 2^{-m}.
\]
Since $y_{\zeta'_T(f(j_1, k_1), \ldots, f(j_n, k_n))} = T(y_{f(j_1, k)}, \ldots, y_{f(j_n, k)})$, it follows from the continuity of $T$ and $\Phi$ that 
\[
d(\Phi(x_{\zeta_T(j_1, \ldots, j_n)}), T(\Phi(x_{j_1}), \ldots, \Phi(x_{j_n}))) \leq 2^{-m},
\]
but $x_{\zeta_T(j_1, \ldots, j_n)}) = T(x_{j_1}, \ldots, x_{j_n})$.  

It similarly follows from condition (\ref{itm:def.R.Func.Pres}) that if $F$ is an $n$-ary functional of $\M$, then 
$F(p_1, \ldots, p_n) = F(\Phi(p_1), \ldots, \Phi(p_n))$ for all points $p_1, \ldots, p_n$ of $\M$. 
Finally, it follows from (\ref{itm:def.R.Const.Pres}) that if $c$ is a constant of $\M$, then $\Phi(c) = c$. 
\end{proof}

We note that the proof of Lemma \ref{lm:pi.01} is uniform in that an index of $\cR$ can be computed uniformly from indices of 
$\M^\#$ and $\M^+$.  

\begin{proof}[Proof of Theorem \ref{thm:low.ii}]
Suppose $\d$ is low for isomorphism.  
Suppose also that $\M^\#$ and $\M^+$ are computable presentations of a metric structure so that $\d$ computes an isometric isomorphism of $\M^\#$ onto $\M^+$.  Let $\cR$ be a $\Pi_1^0$ class as given by Lemma \ref{lm:pi.01}.  
Then, by Theorem 4 of \cite{ft-lowpaths}, $\d$ computes a point of $\cR$.  
Thus, $\d$ computes an isometric isomorphism of $\M^\#$ onto $\M^+$.  

Conversely, suppose $\d$ is low for isometric isomorphism of metric structures.  Again, since every countable algebraic structure
can be represented as a metric structure, it follows that $\d$ is low for isometric isomorphism of countable algebraic structures.
\end{proof}

\section{Lowness for isometry of metric spaces}\label{sec:metric}

The class of metric spaces is, of course, the class of $\mathcal{S}$-structures where $\mathcal{S}$ has no operation or functional symbols.  Thus, when we apply Definition \ref{def:low.ii} to these structures, we omit the term `isomorphism', and we state our main result as follows.

\begin{thm}\label{thm:low.ii.metric}
Every Turing degree is low for isomorphism if and only if it is low for isometry.
\end{thm}

\begin{proof}
It follows from Theorem \ref{thm:low.ii} that every Turing degree that is low for isomorphism is low for isometry.
Suppose that $\mathbf{d}$ is low for isometry. We use Melnikov's technique of representing a graph as a metric space \cite{Melnikov.2013}; this will suffice because graphs are universal structures \cite{hkss02}. We will represent undirected graphs with no loops as metric spaces as follows. Suppose that $G=(V,E)$ is such a graph. We define
$$
d_G(v_0,v_1)=\begin{cases}
0 & v_0=v_1 \\
1 & (v_0,v_1)\in E \\
2 & \mbox{else}
\end{cases}.
$$
This is clearly a metric, so we can write $M(G)$ for the metric space given by $(G,d_G)$. Now suppose that the graph $G_0=(V_0,E_0)$ is ${\mathbf{d}}$-isomorphic to the graph $G_1=(V_1,E_1)$. Since ${\mathbf{d}}$ can compute an isomorphism from $G_0$ to $G_1$,  ${\mathbf{d}}$ can clearly compute an isometry between $M(G_0)$ to $M(G_1)$. However, since $\mathbf{d}$ is low for isometry, there is a computable isometry from $M(G_0)$ to $M(G_1)$. Since $\mathbf{0}$ can compute a distance-preserving function from $M(G_0)$ to $M(G_1)$, there is a computable function that maps each pair of vertices $(v_0,v_1)$ in $G_0$ to another pair of points in $G_1$ with the same distance between them, that is, another pair of points with the same edge-relation (identical, connected by an edge, or not connected by an edge). This computable function will give us a graph isomorphism from $G_0$ to $G_1$.
\end{proof}

This theorem allows us to make some observations based on Franklin and Solomon's work on the degrees that are low for isomorphism: every 2-generic is low for isomorphism and thus low for isometry and isometric isomorphism, there are hyperimmune-free degrees that are low for isomorphism and thus low for isometry and isometric isomorphism, etc.\ \cite{fs-lowim}. 

\section{Results on Banach spaces}\label{sec:banach}

Let $\F_\Q = \F \cap \Q(i)$.  We refer to the elements of $\F_\Q$ as \emph{rational scalars}.

Let $\LB$ denote the metric signature of Banach spaces, which consists of a binary operation symbol $\mathquote{+}$, 
a unary operation symbol $\mathquote{\cdot_s}$ for each rational scalar $s$, a unary functional symbol $\mathquote{\norm{\ }}$, and a constant symbol $\mathquote{\zerovec}$.  Clearly, $\LB$ is computable.

Let $\B$ be a Banach space.  Then $\B$ can be represented as the interpretation of $\LB$ in which $\mathquote{+}$ is interpreted as vector addition, $\mathquote{\cdot_s}$ is interpreted as multiplication by the scalar $s$, $\mathquote{\norm{\ }}$ is interpreted as the norm of $\B$, and $\mathquote{\zerovec}$ is interpreted as the zero vector of $\B$.  
There is no loss of generality due to the restriction to rational scalars.  In particular, any map that preserves multiplication by rational 
scalars also preserves multiplication by scalars.

If $\B^\#$ is a presentation of a Banach space $\B$, then the rational points of $\B^\#$ are precisely the rational linear combinations of distinguished points of $\B^\#$, i.e., vectors that can be expressed in the form $\sum_{j \leq M} \alpha_j v_j$ where $\alpha_j \in \F_\Q$ and each $v_j$ is a rational vector of $\B^\#$.  

The following is an immediate consequence of Lemma \ref{lm:pi.01}.

\begin{thm}\label{thm:low.ii.Banach}
Every Turing degree that is low for isomorphism is also low for 
isometric isomorphism of Banach spaces.
\end{thm}    

The main obstacle to proving the converse of Theorem \ref{thm:low.ii.Banach} is the apparent lack of a method of effectively encoding members of a sufficiently universal class of countable algebraic structures into Banach spaces.  For example, the proof of Theorem \ref{thm:low.ii.metric} turns on a technique for representing graphs as metric spaces.  We are not aware of any such method for representing graphs as Banach spaces.  The closest things we are aware of are the techniques for encoding well-founded trees into Banach spaces in \cite{Ferenczi.Louveau.Rosendal.2009}.  However, the class of well-founded trees is not sufficiently universal (in the sense discussed in the proof of Theorem \ref{thm:low.ii.Banach}).

\section{Results on Lebesgue spaces}\label{sec:lebesgue}

We begin by recalling the definition of $L^p(\Omega)$.

\begin{definition}\label{def:Lp}
Let $\Omega = (X, \mathcal{S}, \mu)$ be a measure space, and suppose $1 \leq p < \infty$.
Then $L^p(\Omega)$ is the set of all measurable $f : X \rightarrow \F$ such that 
	$\int |f|^p\ d\mu < \infty$.  
\end{definition}

If $1 \leq p < \infty$, then $L^p(\Omega)$ is a Banach space under the norm 
	\[
	\norm{f}_p = \left( \int |f|^p\ d\mu \right)^{1/p}
	\]
provided we identify functions that agree almost everywhere.  Thus, a vector in $L^p(\Omega)$ is not a function but an equivalence class of functions.

A Banach space $\B$ is an \emph{$L^p$-space} if there is a measure space $\Omega$ so that 
$L^p(\Omega) = \B$.  A Banach space is a \emph{Lebesgue space} if it is an $L^p$-space for some $p$.

We do not consider $L^\infty$-spaces since no infinite-dimensional $L^\infty$-space is separable, and our treatment of computability on Banach spaces presumes separability.  

Particular $L^p$-spaces of interest are the following. 

\begin{definition}\label{def:Lplplpn}
\begin{enumerate}
	\item $L^p[0,1] = L^p([0,1], \mathcal{S}, m)$ where $\mathcal{S}$ is the $\sigma$-algebra of Lebesgue measurable subsets of $[0,1]$ and $m$ is the Lebesgue measure on $[0,1]$.
	
	\item $\ell^p = L^p(\N, \mathcal{P}(\N), \mu)$ where $\mu$ is the counting measure on $\N$. 
	
	\item $\ell^p_n = L^p(\{1, \ldots, n\}, \mathcal{P}(\{1, \ldots, n\}), \mu)$ where $\mu$ is the counting measure on $\{1, \ldots, n\}$.
\end{enumerate} 
\end{definition}

When $V_0$ and $V_1$ are vector spaces, we let $V_0 \oplus V_1$ denote their external 
direct sum.  Suppose $\strb_0$ and $\strb_1$ are Banach spaces.  
Then $\strb_0 \oplus_p \strb_1$ consists of the vector space $\strb_0 \oplus \strb_1$ together with the 
norm defined by 
\[
\norm{(v_0,v_1)}^p = \norm{v_0}_{\strb_0}^p + \norm{v_1}_{\strb_1}^p.
\]
$\strb_0 \oplus_p \strb_1$ is called the \emph{$L^p$-sum} of $\strb_0$ and $\strb_1$ and is a Banach space in its own right.  

It is well known that every nonzero $L^2$-space is isometrically isomorphic to $\ell^2$ or to $\ell^2_n$ for some $n$.  For $p \neq 2$, we can classify all nonzero separable $L^p$-spaces using the three Banach spaces defined in Definition \ref{def:Lplplpn} or their $L^p$-sums.  A proof of the following theorem can be found in \cite{Cembranos.Mendoza.1997}.

\begin{thm}[Classification of separable $L^p$-spaces]\label{thm:classification}
Suppose $1 \leq p < \infty$ and $p \neq 2$.  
Then every nonzero separable $L^p$-space is isometrically isomorphic to exactly one of the following.
\begin{enumerate}
	\item $\ell^p_n$ for some $n \geq 1$.  In this case, the underlying measure space is purely atomic and has exactly $n$ atoms.
	
	\item $\ell^p$.  In this case, the underlying measure space is purely atomic and has $\aleph_0$ atoms.
	
	\item $L^p[0,1]$.  In this case, the underlying measure space is nonatomic.
	
	\item $\ell^p_n \oplus_p L^p[0,1]$ for some $n \geq 1$.  In this case, the underlying measure space
	has exactly $n$ atoms but is not purely atomic.
	
	\item $\ell^p \oplus_p L^p[0,1]$.  In this case, the underlying measure space
	has $\aleph_0$ atoms but is not purely atomic.
\end{enumerate}
\end{thm}

Let $1 \leq p < \infty$.  The standard presentations of $\ell^p_n$ and $\ell^p$ are given by the standard bases for these 
spaces.  Let $\{D_n\}_{n \in \N}$ be a standard enumeration of the dyadic subintervals of $[0,1]$.  
The standard presentation of $L^p[0,1]$ is the presentation in which the $n$-th distinguished point is 
$\Ind_{D_n}$, or the characteristic (indicator) function of $D_n$.   The standard presentations of $\ell^p_n \oplus L^p[0,1]$ and $\ell^p \oplus L^p[0,1]$ are defined 
accordingly.

Our results on lowness for isometric isomorphism of these structures are the following.

\begin{thm}\label{thm:low.4.ii.lp}
Suppose $1 \leq p < \infty$ is computable and $p \neq 2$.  
Then every Turing degree is low for isometric isomorphism of $\ell^p$ if and only if it does not bound a c.e. degree.
\end{thm}

\begin{thm}\label{thm:low.4.ii.lpn.Lp01}
A Turing degree is low for $\ell^p_n \oplus_p L^p[0,1]$-isometric isomorphism if and only if 
it does not bound a c.e. degree.
\end{thm}

\begin{thm}\label{thm:low.4.ii.lp.Lp01}
A Turing degree is low for $\ell^p \oplus_p L^p[0,1]$-isometric isomorphism if and only if it does not bound a $\Sigma_2^0$ degree.
\end{thm}

Theorem \ref{thm:low.4.ii.lp} is an immediate result of the result of Stull and McNicholl that when $p \geq 1$ is computable and not $2$, the degrees of isometric isomorphism of $\ell^p$ are precisely the c.e. Turing degrees.

We now discuss the apparatus from prior work used to prove Theorems \ref{thm:low.4.ii.lp} through \ref{thm:low.4.ii.lp.Lp01}.

Vectors $f,g \in L^p(\Omega)$ are \emph{disjointly supported} if $f \cdot g = \zerovec$. It follows from a result of J. Lamperti that when $p \neq 2$, every isometric endomorphism of an $L^p$-space preserves disjointness of support \cite{Lamperti.1958}.  

We order the vectors of an $L^p$-space as follows.  When $f,g \in L^p(\Omega)$, $f$ is said to be a \emph{component} of $g$ if $f = g \cdot \Ind_A$ for some measurable $A$, where $\Ind_A$ is the characteristic (indicator) function of $A$.  We write $f \preceq g$ if $f$ is a component of $g$.  
Note that $f \preceq g$ if and only if $g - f$ and $f$ are disjointly supported.  Hence, when $p \neq 2$, every isometric endomorphism of an $L^p$-space also preserves $\preceq$.  Note also that 
$f$ is an atom of $\preceq$ if and only if the support of $f$ is an atom of $\Omega$.

If $\strb$ is a Banach space, then a \emph{vector tree} of $\strb$ is an injective map from a subtree of $\N^{< \N}$ into $\strb$.  Suppose $\phi$ is a vector tree of $\strb$, and let $S = \dom(\phi)$.  We say that each vector in $\ran(\phi)$ is a \emph{vector of $\phi$}. We further say $\phi$ is 
\begin{itemize}
\item \emph{nonvanishing} if $\zerovec$ is not a vector of $\phi$,
\item \emph{linearly dense} if its range is linearly dense, and
\item \emph{summative} if 
for every nonterminal node $\nu$ of $\dom(\phi)$, $\phi(\nu) = \sum_{\nu'} \phi(\nu')$ where
$\nu'$ ranges over the children of $\nu$ in $S$.  
\end{itemize}
Additionally, if $\strb$ is an $L^p$-space, we say $\phi$ is \emph{separating} if it always maps incomparable nodes to disjointly supported vectors.  
Finally, we say $\phi$ is a \emph{disintegration} if it is nonvanishing, separating, summative, and linearly dense.

Fix a disintegration $\phi$ of an $L^p$-space.  
A nonroot node $\nu$ of $S$ is an \emph{almost norm-maximizing} child of its parent if 
\[
\norm{\phi(\nu')}_p^p \leq \norm{\phi(\nu)}_p^p + 2^{-|\nu|}
\]
whenever $\nu' \in S$ is a sibling of $\nu$, and a chain $C \subseteq S$ is \emph{almost norm-maximizing} if 
for every $\nu \in C$, if $\nu$ has a child in $S$, then $C$ contains an almost norm-maximizing child of $\nu$.

The following theorem was first proven for $\ell^p$-spaces in \cite{McNicholl.2017} and generalized to arbitrary $L^p$-spaces in \cite{Brown.McNicholl.2019}.

\begin{thm}\label{thm:limits.chains}
Suppose $\phi$ is a disintegration of $L^p(\Omega)$.
\begin{enumerate}
	\item If $C \subseteq \dom(\phi)$ is an almost norm-maximizing chain, then the $\preceq$-infimum of $\phi[C]$ exists and is either \textbf{0} or an atom of $\preceq$.  Furthermore, the $\preceq$-infimum of $\phi[C]$ is the limit in the $L^p$-norm of $\phi(\nu)$ as $\nu$ traverses the nodes in $C$ in increasing order.\label{thm:limits.chains::itm:inf}
		
	\item If $\{C_n\}_{n < \kappa}$ is a partition of $\dom(\phi)$ into almost norm-maximizing chains (where $\kappa \leq \N$), then the $\preceq$-infima of $\phi[C_0], \phi[C_1], ...$ are disjointly supported.  Furthermore, if $A$ is an atom of $\Omega$, then there exists a unique $n$ so that $A$ is the support of the $\preceq$-infimum of $\phi[C_n]$. \label{thm:limits.chains::itm:unique} 
\end{enumerate}
\end{thm}

\begin{thm}\label{thm:disint.comp}
Suppose $p \geq 1$ is computable and $p \neq 2$.  Then every computable presentation of a nonzero $L^p$-space has a computable disintegration.
\end{thm}

\begin{thm}\label{thm:comp.partition}
If $\strb^\#$ is a computable presentation of an $L^p$-space, and if $\phi$ is a computable disintegration of $\strb^\#$, then there is a partition $\{C_n\}_{n < \kappa}$ of $\dom(\phi)$ into uniformly c.e. almost norm-maximizing chains (where $\kappa \leq \N$).
\end{thm}

\subsection{Proof of Theorem \ref{thm:low.4.ii.lpn.Lp01}}\label{sec:Lebesgue::subsec:lpn.Lp01}

It suffices to show the following.

\begin{thm}\label{thm:deg.isom.lpn.Lp01}
Suppose $p \geq 1$ is computable and $p \neq 2$.  
Then the degrees of isometric isomorphism for $\ell_n^p \oplus_p L^p[0,1]$ are precisely the c.e. degrees.
\end{thm}

Suppose $p \geq 1$ is computable and $p \neq 2$.  One direction of Theorem \ref{thm:deg.isom.lpn.Lp01} is proven in \cite{Brown.McNicholl.2019}; namely, 
every c.e. degree is a degree of isometric isomorphism for $\ell_n^p \oplus_p L^p[0,1]$.
Thus, we need only show that every degree of isometric isomorphism for $\ell^p_n \oplus_p L^p[0,1]$ is 
c.e..

 Let $\strb = \ell^p_n \oplus_p L^p[0,1]$.  Let $P$ denote the projection of $\strb$ onto its embedded copy of $L^p[0,1]$; i.e., $P(u,v) = (0,v)$.  
Finally, let $I$ denote the identity map on $\strb$.

Suppose $\strb^\#$ is a computable presentation of $\strb$.  By Theorem \ref{thm:disint.comp}, there is a computable disintegration $\phi$ of $\strb^\#$; let $S = \dom(\phi)$.  By Theorem \ref{thm:comp.partition}, there is a partition 
$\{C_j\}_{j = 0}^\infty$ of $S$ into uniformly c.e.\ almost norm-maximizing chains.  
Let $g_j = \lim_{\nu \in C_j} \phi(\nu)$.  By Theorem \ref{thm:limits.chains}, there are exactly $n$ values of $j$ so that $g_j \neq \zerovec$; let $j_1, \ldots, j_n$ denote these values.  Again, by Theorem \ref{thm:limits.chains}, each $g_{j_s}$ is an atom of $\preceq$.  So $P(g_{j_s}) = \zerovec$, and for each $k \in \{1, \ldots, n\}$ there is exactly one $s$ so that $\{k\} = \supp((I - P)(g_{j_s}))$; without loss of generality, assume $s = k$.  

Let $\mathbf{d}$ be the Turing degree of the join of the right Dedekind cuts of $\norm{g_{j_1}}_p$, $\ldots$, $\norm{g_{j_n}}_p$.  
Thus, $\mathbf{d}$ is c.e..  We show that $\mathbf{d}$ is the degree of isometric isomorphism of 
$\ell^p_n \oplus_p L^p[0,1]$.  

We first claim that $g_{j_s}$ is a $\mathbf{d}$-computable vector of $\strb^\#$.  
We see that $\norm{g_{j_s}}_p$ is a $\mathbf{d}$-computable real.  
If $\nu \in C_{j_s}$, then by Theorem \ref{thm:limits.chains}.\ref{thm:limits.chains::itm:inf}, $g_{j_s} \preceq \phi(\nu)$, and so 
$\norm{\phi(\nu) - g_{j_s}}_p^p = \norm{\phi(\nu)}_p^p - \norm{g_{j_s}}_p^p$.  
So for each $\nu \in C_{j_s}$, $\norm{\phi(\nu) - g_{j_s}}_p$ is $\mathbf{d}$-computable uniformly in $\nu$.  Again, by Theorem \ref{thm:limits.chains}.\ref{thm:limits.chains::itm:inf}, for each $k \in \N$, there is a $\nu \in C_{j_s}$ such that 
$\norm{\phi(\nu) - g_{j_s}}_p < 2^{-k}$.  The $\mathbf{d}$-computability of $g_{j_s}$ as a vector of $\strb^\#$ now follows.

For each $\nu \in S$, $g_{j_s} \preceq \phi(\nu)$ if and only if $\nu \in C_{j_s}$.  Thus, for each $\nu \in S$, 
\[
P(\phi(\nu)) = \phi(\nu) - \sum_{\nu \in C_{j_s}} g_{j_s}.
\]
Let $\{\nu_t\}_{t \in \N}$ be an effective enumeration of $S$, and let $L^p[0,1]^\#$ be the presentation of $L^p[0,1]$ whose $t$-th distinguished vector is $P(\phi(\nu_t))$.  This presentation is $\mathbf{d}$-computable.  So, by the relativization of Theorem \ref{thm:disint.comp}, there is a $\mathbf{d}$-computable isometric isomorphism $T_1$ of $L^p[0,1]$ onto $L^p[0,1]^\#$.  Let $T: \strb \rightarrow \strb^\#$ be defined by
\[
T\left(\sum_{k =0}^{n-1} \alpha_k e_k, f\right) = \sum_{k = 0}^{n-1} \frac{\alpha_k}{\norm{g_{j_k}}_p} g_{j_k} + T_1(f).
\]
Thus, $T$ is a $\mathbf{d}$-computable isometric isomorphism of $\strb$ onto $\strb^\#$.  

Now suppose $\mathbf{d}_1$ computes an isometric isomorphism $T'$ of $\strb$ onto $\strb^\#$.  
As noted above, $T'$ preserves disjointness of support and $\preceq$.  Thus, for each $k$, there is an $s_k$ so that 
$T'(e_k, \zerovec) = j_{s_k}$.  Then $\mathbf{d}_1$ computes $\norm{g_{j_s}}_p$ from $s$, and so 
$\mathbf{d}_1 \geq_T \mathbf{d}$.  Thus, $\mathbf{d}$ is the degree of isometric isomorphism of $\ell^p_n \oplus_p L^p[0,1]$.

\subsection{Proof of Theorem \ref{thm:low.4.ii.lp.Lp01}}\label{sec:Lebesgue::subsec:lp.Lp01}

It suffices to prove the following.  

\begin{thm}\label{thm:deg.isom.lpLp01}
Suppose $p \geq 1$ is computable.  Then every computable presentation of 
$\ell^p \oplus_p L^p[0,1]$ has a $\Sigma_2^0$ degree of isometric isomorphism.  If $p \neq 2$, then 
every $\Sigma_2^0$ degree is the degree of isometric isomorphism of a computable presentation 
of $\ell^p \oplus_p L^p[0,1]$.
\end{thm}

Without loss of generality, assume $p \neq 2$.  Let $\strb = \ell^p \oplus_p L^p[0,1]$, and let 
$P$ denote the projection of $\strb$ onto its embedded copy of $L^p[0,1]$.  Suppose 
$\strb^\#$ is a computable presentation of $\strb$.  Let $\phi$ be a disintegration of 
$\strb^\#$, and let $\{C_n\}_{n = 0}^\infty$ be a partition of $S = \dom(\phi)$ into uniformly c.e. almost norm-maximizing chains.  Set $g_n = \lim_{\nu \in C_n} \phi(\nu)$.  

We now define two sets:
\begin{eqnarray*}
A_1 & = & \{\langle n, k \rangle\ :\ \norm{g_n}_p \geq 2^{-k}\}\\
A_2 & = & \left\{ \langle \nu, M, k \rangle\ :\ \Big\|\sum_{n \geq M} \chi_{C_n}(\nu) g_n\Big\|_p \leq 2^{-k}\right\}
\end{eqnarray*}
This allows us to define $A  =  A_1 \oplus A_2$ and $\mathbf{a} = \deg(A)$.

We first claim that $\mathbf{a}$ is the degree of isometric isomorphism of $\strb^\#$.  
To this end, we first show that $\mathbf{a}$ computes an isometric isomorphism of $\strb$ onto $\strb^\#$.  
We begin by noting that $A_1$ computes an enumeration of all $n\in \N$ so that $g_n \neq \zerovec$.  
It follows from Theorem \ref{thm:limits.chains} that $A_1$ computes a linear isometric map $T_1$ of 
$\ell^p$ into $\strb^\#$ so that $\ran(T_1) = (I - P)[\strb]$ (where $I$ denotes the identity map).  It follows as in the proof of Theorem \ref{thm:deg.isom.lpn.Lp01} that $A_2$ computes a linear isometric map $T_2$ of $L^p[0,1]$ onto $P[\strb]$.  Thus, 
$\mathbf{a}$ computes an isometric isomorphism of $\strb$ onto $\strb^\#$.  

Now, suppose $\mathbf{b}$ computes an isometric isomorphism of $\strb$ onto $\strb^\#$.  
We can assume that for each nonterminal node $\nu$ of $S$, if $\nu \in C_n$, then 
$C_n$ contains a child $\nu'$ of $\nu$ so that $\norm{\phi(\mu)}_p^p < \norm{\phi(\nu')}_p^p + \frac{1}{2} \norm{\phi(\nu)}_p^p$ for every child $\mu$ of $\nu$ in $S$.  It then follows, by the same argument in the proof of Lemma 6.2 of \cite{McNicholl.Stull.2019}, that $\{\norm{g_n}_p\}_{n = 0}^\infty$ is a 
$\mathbf{b}$-computable sequence of reals.  Thus, $\mathbf{b}$ computes an enumeration of all $n \in \N$ so that $g_n \neq \zerovec$.   From this, we conclude that $\mathbf{b}$ computes $A_1$.  
  Since $\mathbf{b}$ computes $T^{-1}$, $\mathbf{b}$
also computes $P = T P T^{-1}$.  Since 
\[
P(\phi(\nu)) = \phi(\nu) - \sum_n \chi_{C_n}(\nu) g_n,
\]
it follows that $\mathbf{b}$ computes $A_2$ as well.  Hence, $\mathbf{a}$ is the degree of isometric isomorphism of $\strb^\#$.  

We now show that $\mathbf{a}$ is $\Sigma_2^0$.  We first note that 
\[
\langle n,k \rangle\ \in A_1 \Leftrightarrow\ \forall q \in \Q [q > \norm{g_n}_p\ \Rightarrow\ 2^{-k} \leq q].
\]
Since $\norm{g_n}_p$ is right-c.e. uniformly in $n$, $A_1$ is $\Pi_1^0$.  
Since $\phi$ is separating, $g_0$, $g_1$, $\ldots$ are disjointly supported.  
Thus, 
\begin{eqnarray*}
\Big\|\sum_{n \geq M} \chi_{C_n}(\nu) g_n \Big\|_p \leq 2^{-k} & \Leftrightarrow & \sum_{n \geq M} \norm{g_n}_p^p \chi_{C_n}(\nu) \leq 2^{-kp}\\
& \Leftrightarrow & \forall k' \forall M' \geq M \exists q_M, \ldots, q_{M'} \in \Q [ \forall M \leq j \leq M'\ q_j > \norm{g_n}_p^p\\
& & \wedge\ \sum_{n = M}^{M'} q_n \chi_{C_n}(\nu) < 2^{-kp} + 2^{-k'}].
\end{eqnarray*}
Therefore, $A_2$ is $\Pi_2^0$ and so $\mathbf{a}$ is $\Sigma_2^0$.  

\section{Conclusion}\label{sec:conclusion}

Our goal in this paper has been to extend the investigation of lowness for isomorphism to lowness for isometric isomorphism of 
metric structures and for particular classes of metric structures.  We have produced a framework for this extension that 
naturally extends the framework for countable algebraic structures, and we have 
obtained several initial results in this new direction.  In particular, we have identified the degrees that are low for isometry as precisely the degrees that are low for isomorphism. While these degrees have no known full characterization themselves, this is one of the most robust lowness notions that has been studied: the degrees that are low for isomorphism are precisely those that are low for paths and, now, those that are low for isometry.

Our conclusions so far suggest a number of questions.  
To begin, our result on Banach spaces (Theorem \ref{thm:low.ii.Banach}) leads to the following.

\begin{question}\label{ques:low.ii.Banach}
Which Turing degrees are low for isometric isomorphism of Banach spaces?
\end{question}

On one hand, it seems reasonable to conjecture that every degree that is low for isometric isomorphism of Banach spaces is 
low for isomorphism.  
On the other hand, one possible way to differentiate the two might be as follows. It is known that the degrees that are low for isomorphism have measure zero; in fact, no Martin-L\"{o}f random degree is low for isomorphism. If the degrees that are low for isometric isomorphism were found to have measure greater than zero, Question \ref{ques:low.ii.Banach} would be answered immediately.  Thus, we are led to the following.

\begin{question}
Do the degrees that are low for isometric isomorphism of Banach spaces have measure zero?
\end{question}

For $p \geq 1$ computable, we have characterized the degrees that are low for isometric isomorphism of the standard copies of the separable $L^p$-spaces.  However, we have not determined the degrees that are low for isometric isomorphism of the $L^p$-spaces in general or of specific types of $L^p$-spaces rather than only for their standard presentations.  These considerations suggest the following. 

\begin{question}\label{ques:low.ii.Lp}
Suppose $p$ is a computable real so that $p \geq 1$ and $p \neq 2$. 
\begin{enumerate} 
	\item Which Turing degrees are low for isometric isomorphism of $L^p$-spaces?  
	
	\item Which Turing degrees are low for isometric isomorphism of $\ell^p$-spaces (i.e., spaces that are isometrically isomorphic to $\ell^p$)?
	
	\item Which Turing degrees are low for isometric isomorphism of $\ell^p_n \oplus_n L^p[0,1]$-spaces?  
	
	\item Which Turing degrees are low for isometric isomorphism of $\ell^p \oplus_n L^p[0,1]$-spaces?  
\end{enumerate}
\end{question}

Our motivation for investigating $L^p$-spaces is their importance in the study of Banach spaces. In particular, it is well-known that 
every separable Banach space is isometrically isomorphic to a quotient of $\ell^1$.   
Also important are the $C(X)$-spaces since, for example, every separable Banach space isometrically embeds into $C[0,1]$.  
This suggests another line of inquiry. 

\begin{question}\label{ques:low.ii.C01}
Which Turing degrees are low for isometric isomorphism of $C[0,1]$?
\end{question}

We conclude with a question about degrees of isomorphism.  We recall that a Turing degree $\mathbf{d}$ is a \emph{degree of categoricity} if there is a computable structure $\mathcal{A}$ that is $\mathbf{c}$-computably categorical if and only if $\mathbf{c}\geq_T \mathbf{d}$ \cite{fkm10}. Furthermore, a degree of categoricity $\mathbf{d}$ is \emph{strong} if there is a computable structure $\mathcal{A}$ with computable copies $\mathcal{A}_1$ and $\mathcal{A}_2$ such that not only does $\mathcal{A}$ have degree of categoricity $\mathbf{d}$, every isomorphism from $\mathcal{A}_1$ to $\mathcal{A}_2$ computes $\mathbf{d}$.

This class of degrees does not have a full characterization, either. However, it is known that all such degrees are hyperarithmetic \cite{cfs13} and that every degree $\mathbf{d}$ that is both c.e.\ or d.c.e.\ in $\mathbf{0}^{(m)}$ and Turing above $\mathbf{0}^{(m)}$ is a strong degree of categoricity \cite{fkm10}; in fact, this is true for $\mathbf{0}^{(\alpha)}$ for any successor ordinal $\alpha$ \cite{cfs13}.

\begin{question}
Is every degree of isomorphism a degree of categoricity? If so, is every degree of isomorphism a strong degree of categoricity?
\end{question}


\begin{thebibliography}{10}

\bibitem{Ben-Yaacov.Berenstein.Henson.Usvyatsov.2008}
Ita\"{i} Ben~Yaacov, Alexander Berenstein, C.~Ward Henson, and Alexander
  Usvyatsov.
\newblock Model theory for metric structures.
\newblock In {\em Model theory with applications to algebra and analysis.
  {V}ol. 2}, volume 350 of {\em London Math. Soc. Lecture Note Ser.}, pages
  315--427. Cambridge Univ. Press, Cambridge, 2008.

\bibitem{Brown.McNicholl.2019}
T.~Brown and T.~H. McNicholl.
\newblock Analytic computable structure theory and $l^p$-spaces part 2.
\newblock To appear in Archive for Mathematical Logic. Preprint available at
  http://arxiv.org/abs/1801.00355, 2019.

\bibitem{Cembranos.Mendoza.1997}
Pilar Cembranos and Jos{{\'e}} Mendoza.
\newblock {\em Banach spaces of vector-valued functions}, volume 1676 of {\em
  Lecture Notes in Mathematics}.
\newblock Springer-Verlag, Berlin, 1997.

\bibitem{Clanin.McNicholl.Stull.2019}
Joe Clanin, Timothy~H. McNicholl, and Don~M. Stull.
\newblock Analytic computable structure theory and {$L^p$} spaces.
\newblock {\em Fund. Math.}, 244(3):255--285, 2019.

\bibitem{cfs13}
Barbara~F. Csima, Johanna~N.Y. Franklin, and Richard~A. Shore.
\newblock Degrees of categoricity and the hyperarithmetic hierarchy.
\newblock {\em Notre Dame J. Formal Logic}, 54(2):215--231, 2013.

\bibitem{Ferenczi.Louveau.Rosendal.2009}
Valentin Ferenczi, Alain Louveau, and Christian Rosendal.
\newblock The complexity of classifying separable {B}anach spaces up to
  isomorphism.
\newblock {\em J. Lond. Math. Soc. (2)}, 79(2):323--345, 2009.

\bibitem{fkm10}
Ekaterina~B. Fokina, Iskander Kalimullin, and Russell Miller.
\newblock Degrees of categoricity of computable structures.
\newblock {\em Arch. Math. Logic}, 49(1):51--67, 2010.

\bibitem{fs-lowim}
Johanna~N.Y. Franklin and Reed Solomon.
\newblock Degrees that are low for isomorphism.
\newblock {\em Computability}, 3(2):73--89, 2014.

\bibitem{ft-lowpaths}
Johanna~N.Y. Franklin and Dan Turetsky.
\newblock Taking the path computably traveled.
\newblock {\em Journal of Logic and Computation}.
\newblock To appear.

\bibitem{hkss02}
D.~Hirschfeldt, B.~Khoussainov, R.~Shore, and A.~Slinko.
\newblock Degree spectra and computable dimensions in algebraic structures.
\newblock {\em Ann. Pure Appl. Logic}, 115(1-3):71--113, 2002.

\bibitem{Lamperti.1958}
John Lamperti.
\newblock On the isometries of certain function-spaces.
\newblock {\em Pacific J. Math.}, 8:459--466, 1958.

\bibitem{McNicholl.Stull.2019}
T.~H. McNicholl and D.~M. Stull.
\newblock The isometry degree of a computable copy of $\ell^p$.
\newblock {\em Computability}, 8(2):179 -- 189, 2019.

\bibitem{McNicholl.2017}
Timothy~H. McNicholl.
\newblock Computable copies of $\ell^p$.
\newblock {\em Computability}, 6(4):391 -- 408, 2017.

\bibitem{Melnikov.2013}
Alexander~G. Melnikov.
\newblock Computably isometric spaces.
\newblock {\em J. Symbolic Logic}, 78(4):1055--1085, 2013.

\bibitem{Melnikov.Ng.2014}
Alexander~G. Melnikov and Keng~Meng Ng.
\newblock Computable structures and operations on the space of continuous
  functions.
\newblock {\em Fundamenta Mathematicae}, 233(2):1--41, 2014.

\bibitem{Melnikov.Nies.2013}
Alexander~G. Melnikov and Andr{\'e} Nies.
\newblock The classification problem for compact computable metric spaces.
\newblock In {\em The nature of computation}, volume 7921 of {\em Lecture Notes
  in Comput. Sci.}, pages 320--328. Springer, Heidelberg, 2013.

\bibitem{Nies.Solecki.2015}
Andr{\'e} Nies and Slawomir Solecki.
\newblock Local compactness for computable {P}olish metric spaces is
  {$\Pi^1_1$}-complete.
\newblock In {\em Evolving computability}, volume 9136 of {\em Lecture Notes in
  Comput. Sci.}, pages 286--290. Springer, 2015.

\bibitem{Pour-El.Richards.1989}
Marian~B. Pour-El and J.~Ian Richards.
\newblock {\em Computability in analysis and physics}.
\newblock Perspectives in Mathematical Logic. Springer-Verlag, Berlin, 1989.

\bibitem{suggs}
Jacob Suggs.
\newblock {\em Degrees that are low for $\mathcal{C}$ isomorphism}.
\newblock PhD thesis, University of Connecticut, 2015.

\end{thebibliography}
\end{document}